\documentclass{article}%
\usepackage{amsmath}
\usepackage{amsfonts}
\usepackage{amssymb}
\usepackage{graphicx}%
\providecommand{\U}[1]{\protect\rule{.1in}{.1in}}
\newtheorem{theorem}{Theorem}

\newtheorem{definition}[theorem]{Definition}

\newtheorem{lemma}[theorem]{Lemma}

\newtheorem{remark}[theorem]{Remark}

\newenvironment{proof}[1][Proof]{\noindent\textbf{#1.} }{\ \rule{0.5em}{0.5em}}
\begin{document}

\title{Existence of solutions of $\alpha\in\left(  2,3\right]  $ order fractional
three point boundary value problems with integral conditions}
\author{N. I. Mahmudov, S. Unul\\Eastern Mediterranean University\\Gazimagusa, TRNC, Mersin 10, Turkey \\Email: nazim.mahmudov@emu.edu.tr\\$\ \ \ \ \ \ \ \ \ $\ \ sinem.unul@emu.edu.tr}
\date{}
\maketitle

\begin{abstract}
Existence and uniqueness of solutions for $\alpha\in\left(  2,3\right]  $
order fractional differential equations with three point fractional boundary
and integral conditions is discussed. The results are obtained by using
standard fixed point theorems. Two examples are given to illustrate the results.

\end{abstract}

\section{Introduction}

Recently, the theory on existence and uniqueness of solutions of linear and
nonlinear fractional differential equations has attracted the attention of
many authors, see for example, \cite{agar1}-\cite{wang2} and references
therein. Many of the physical systems can better be described by integral
boundary conditions. Integral boundary conditions are encountered in various
applications such as population dynamics, blood flow models, chemical
engineering and cellular systems. Moreover, boundary value problems with
integral boundary conditions constitute a very interesting and important class
of problems. They include two-point, three-point, multi-point and nonlocal
boundary value problems as special cases. The existing literature mainly deals
with first order and second order boundary value problems and there are a few
papers on third order problems.

Shahed \cite{el} studied existence and nonexistence of positive solution of
nonlinear fractional two-point boundary value problem derivative%
\begin{align*}
\mathfrak{D}_{0^{+}}^{\alpha}u(t)+\lambda a\left(  t\right)  f\left(
u(t)\right)   &  =0,\text{ \ }0<t<1;\text{ }2<\alpha<3,\\
u\left(  0\right)   &  =u^{\prime}\left(  0\right)  =u^{\prime}\left(
1\right)  =0,
\end{align*}
where $\mathfrak{D}_{0^{+}}^{\alpha}$ denotes the Caputo derivative of
fractional order $\alpha,$ $\lambda$ is a positive parameter and $a:\left(
0,1\right)  \rightarrow\left[  0,\infty\right)  $ is continuous function.

In \cite{ahmadntoy1}, Ahmad and Ntouyas studied a boundary value problem of
nonlinear fractional differential equations of order $\alpha\in\left(
2,3\right]  $ with anti-periodic type integral boundary conditions:%
\begin{align*}
\mathfrak{D}_{0^{+}}^{\alpha}u(t) &  =f\left(  t,u(t)\right)  ;\text{
\ }0<t<T;\text{ }2<\alpha\leq3,\\
u^{\left(  j\right)  }(0)-\lambda_{j}u^{\left(  j\right)  }(T) &  =\mu_{j}%
{\displaystyle\int\limits_{0}^{T}}
g_{j}(s,u\left(  s\right)  )ds,\ \ j=0,1,2,
\end{align*}
where $\mathfrak{D}_{0^{+}}^{\alpha}$ denotes the Caputo derivative of
fractional order $\alpha$, $u^{\left(  j\right)  }$ denotes $j$-th derivative
of $u$, $f,g_{0},g_{1},g_{2}:\left[  0,T\right]  \times\mathbb{R}%
\rightarrow\mathbb{R}$ are given continuous functions and $\lambda_{j},\mu
_{j}\in\mathbb{R}$ ($\lambda_{j}\neq1$). The same problem for fractional
differential inclusions is considered in \cite{ahmadntoy2}.

Ahmad and Nieto \cite{ahmad6} studied existence and uniqueness results for the
following general three point fractional boundary value problem involving a
nonlinear fractional differential equation of order $\alpha\in\left(
m-1,m\right]  $,%
\begin{align*}
\mathfrak{D}_{0^{+}}^{\alpha}u(t) &  =f\left(  t,u(t)\right)  ;\text{
\ }0<t<T,\ m\geq2,\\
u\left(  0\right)   &  =u^{\prime}\left(  0\right)  =...=u^{\left(
m-2\right)  }\left(  0\right)  =0,\ \ u\left(  1\right)  =\lambda u\left(
\eta\right)  .
\end{align*}

However, very little work have been done on the case when the nonlinearity $f$
depends on the fractional derivative of the unknown function. Su and Zhang
\cite{su}, Rehman et al. \cite{rehman} studied the existence and uniqueness of
solutions for following nonlinear two-point and three point fractional
boundary value problem when the nonlinearity $f$ depends on the fractional
derivative of the unknown function.

In this paper, we investigate the existence (and uniqueness) of solution for
nonlinear fractional differential equations of order $\alpha\in\left(
2,3\right]  $%

\begin{equation}
\mathfrak{D}_{0^{+}}^{\alpha}u(t)=f\left(  t,u(t),\mathfrak{D}_{0^{+}}%
^{\beta_{1}}u(t),\mathfrak{D}_{0^{+}}^{\beta_{2}}u(t)\right)  ;\text{ \ }0\leq
t\leq T;\text{ }2<\alpha\leq3 \label{p1}%
\end{equation}
with the three point and integral boundary conditions%
\begin{equation}
\left\{
\begin{array}
[c]{l}%
a_{0}u(0)+b_{0}u(T)=\lambda_{0}%
{\displaystyle\int\limits_{0}^{T}}
g_{0}(s,u\left(  s\right)  )ds,\ \ \\
a_{1}\mathfrak{D}_{0^{+}}^{\beta_{1}}u(\eta)+b_{1}\mathfrak{D}_{0^{+}}%
^{\beta_{1}}u(T)=\lambda_{1}%
{\displaystyle\int\limits_{0}^{T}}
g_{1}(s,u\left(  s\right)  )ds,\ \ \ 0<\beta_{1}\leq1,\ \ 0<\eta<T,\\
a_{2}\mathfrak{D}_{0^{+}}^{\beta_{2}}u(\eta)+b_{2}\mathfrak{D}_{0^{+}}%
^{\beta_{2}}u(T)=\lambda_{2}%
{\displaystyle\int\limits_{0}^{T}}
g_{2}(s,u\left(  s\right)  )ds,\ \ \ \ 1<\beta_{2}\leq2,
\end{array}
\right.  \label{p2}%
\end{equation}
where $\mathfrak{D}_{0^{+}}^{\alpha}$ denotes the Caputo fractional derivative
of order $\alpha$, $f,g_{0},g_{1},g_{2}$ are continuous functions.

\section{Preliminaries}

Let us recall some basic definitions \cite{sam}-\cite{kilbas}.

\begin{definition}
\label{Def:1}The Riemann Liouville fractional integral of order $\beta$ for
continous function $f:[0,\infty)\rightarrow\mathbb{R}$ is defined as
\[
I_{0^{+}}^{\alpha}f(t)=\frac{1}{\Gamma(\alpha)}\int\limits_{0}^{t}%
(t-s)^{\alpha-1}f(s)ds,\text{ \ }\alpha>0
\]
provided the integral exists.
\end{definition}

\begin{definition}
For $n$-times continously differentiable function $f:[0,\infty)\rightarrow
\mathbb{R}$ \ the Caputo derivative fractional order $\alpha$ is defined as;%
\[
\mathfrak{D}_{0^{+}}^{\alpha}f(t)=\frac{1}{\Gamma(n-\alpha)}\int
\limits_{0}^{t}(t-s)^{n-\alpha-1}f^{(n)}(s)ds;\text{ \ }n-1<\alpha
<n,\ n=\left[  \alpha\right]  +1,
\]
where $\left[  \alpha\right]  $ denotes the integral part of the real number
$\alpha.$
\end{definition}

\begin{lemma}
\label{Lem:1}Let $\alpha>0.$ Then the differential equation $\mathfrak{D}%
_{0^{+}}^{\alpha}f(t)=0$ has solutions%
\[
f(t)=k_{0}+k_{1}t+k_{2}t^{2}+...+k_{n-1}t^{n-1}%
\]
and
\[
I_{0^{+}}^{\alpha}\mathfrak{D}_{0^{+}}^{\alpha}f(t)=f(t)+k_{0}+k_{1}%
t+k_{2}t^{2}+...+k_{n-1}t^{n-1},
\]
here $k_{i}\in\mathbb{R}$ and $i=1,2,3,...,n-1$, $n=\left[  \alpha\right]
+1.$
\end{lemma}

Caputo fractional derivative of order $n-1<\alpha<n$ for $t^{\gamma}$, is
given as%
\begin{equation}
\mathfrak{D}_{0^{+}}^{\alpha}t^{\gamma}=\left\{
\begin{array}
[c]{l}%
\dfrac{\Gamma\left(  \gamma+1\right)  }{\Gamma\left(  \gamma-\alpha+1\right)
}t^{\gamma-\alpha},\ \ \gamma\in\mathbb{N}\ \ \text{and }\gamma\geq
n\ \ \text{or\ \ }\gamma\notin\mathbb{N}\text{\ \ and\ }\gamma>n-1,\\
0,\ \ \ \gamma\in\left\{  0,1,...,n-1\right\}  .
\end{array}
\right.  \label{d1}%
\end{equation}

Assume that $a_{i},b_{i},\lambda_{i}\in\mathbb{R},$ $0<\eta<T,$ $\beta_{0}=0,$
$0<\beta_{1}\leq1$, $1<\beta_{2}\leq2$ and
\[
a_{0}+b_{0}\neq0,\ \ a_{1}\eta^{1-\beta_{1}}+b_{1}T^{1-\beta_{1}}%
\neq0,\ \ a_{i}\eta^{2-\beta_{i}}+b_{i}T^{2-\beta_{i}}\neq0.
\]
For convenience, we set%
\begin{align*}
\mu^{\beta_{1}}  &  :=\frac{\Gamma(3-\beta_{1})}{2\left(  a_{1}\eta
^{2-\beta_{1}}+b_{1}T^{2-\beta_{1}}\right)  },\ \mu^{\beta_{2}}:=\frac
{\Gamma(3-\beta_{2})}{2\left(  a_{2}\eta^{2-\beta_{2}}+b_{2}T^{2-\beta_{2}%
}\right)  },\ \ \ \ \nu^{\beta_{1}}:=\frac{\Gamma(2-\beta_{1})}{a_{1}%
\eta^{1-\beta_{1}}+b_{1}T^{1-\beta_{1}}},\ \\
\omega_{0}  &  :=\dfrac{1}{a_{0}+b_{0}},\ \ \ \omega_{1}\left(  t\right)
:=\nu^{\beta_{1}}\left(  \dfrac{b_{0}}{a_{0}+b_{0}}T-t\right)  ,\ \ \ \\
\omega_{2}\left(  t\right)   &  :=\dfrac{b_{0}T^{2}}{a_{0}+b_{0}}\mu
^{\beta_{2}}-\dfrac{b_{0}T}{a_{0}+b_{0}}\nu^{\beta_{1}}\dfrac{\mu^{\beta_{2}}%
}{\mu^{\beta_{1}}}+\nu^{\beta_{1}}\dfrac{\mu^{\beta_{2}}}{\mu^{\beta_{1}}%
}t-\mu^{\beta_{2}}t^{2}.
\end{align*}

\begin{lemma}
For any $f,g_{0},g_{1},g_{2}\in C\left(  \left[  0,T\right]  ;\mathbb{R}%
\right)  $, the unique solution of the fractional boundary value problem
\begin{gather}
\ \ \ \ \ \ \mathfrak{D}_{0^{+}}^{\alpha}u(t)=f(t);\text{ \ }0\leq t\leq
T,\text{ }2<\alpha\leq3,\label{e1}\\
\left\{
\begin{array}
[c]{l}%
a_{0}u(0)+b_{0}u(T)=\lambda_{0}%
{\displaystyle\int\limits_{0}^{T}}
g_{0}(s)ds,\\
a_{1}\mathfrak{D}_{0^{+}}^{\beta_{1}}u(\eta)+b_{1}\mathfrak{D}_{0^{+}}%
^{\beta_{1}}u(T)=\lambda_{1}%
{\displaystyle\int\limits_{0}^{T}}
g_{1}(s)ds,\ \ \ \ 0<\eta<T,\ \ 0<\beta_{1}\leq1,\\
a_{2}\mathfrak{D}_{0^{+}}^{\beta_{2}}u(\eta)+b_{2}\mathfrak{D}_{0^{+}}%
^{\beta_{2}}u(T)=\lambda_{2}%
{\displaystyle\int\limits_{0}^{T}}
g_{2}(s)ds,\ \ \ \ 1<\beta_{1}\leq2
\end{array}
\right.  \ \label{e2}%
\end{gather}

is given by%
\begin{align*}
u\left(  t\right)   &  =%
{\displaystyle\int\limits_{0}^{t}}
\dfrac{(t-s)^{\alpha-1}}{\Gamma(\alpha)}f(s)ds+%
{\displaystyle\sum\limits_{i=0}^{2}}
\omega_{i}\left(  t\right)  b_{i}%
{\displaystyle\int\limits_{0}^{T}}
\frac{(T-s)^{\alpha-\beta_{i}-1}}{\Gamma(\alpha-\beta_{i})}f(s)ds\\
&  +%
{\displaystyle\sum\limits_{i=1}^{2}}
\omega_{i}\left(  t\right)  a_{i}%
{\displaystyle\int\limits_{0}^{\eta}}
\frac{(\eta-s)^{\alpha-\beta_{i}-1}}{\Gamma(\alpha-\beta_{i})}f(s)ds-%
{\displaystyle\sum\limits_{i=0}^{2}}
\omega_{i}\left(  t\right)  \lambda_{i}%
{\displaystyle\int\limits_{0}^{T}}
g_{i}(s)ds.
\end{align*}

\end{lemma}

\begin{proof}
By Lemma \ref{Lem:1}, for $2<\alpha\leq3$ the general solution of the equation
$\mathfrak{D}_{0^{+}}^{\alpha}u(t)=f(t)$ can be written as
\begin{equation}
u(t)=\frac{1}{\Gamma(\alpha)}%
{\displaystyle\int\limits_{0}^{t}}
(t-s)^{\alpha-1}f(s)ds-k_{0}-k_{1}t-k_{2}t^{2}, \label{ss1}%
\end{equation}
where $k_{0},k_{1},k_{2}\in\mathbb{R}$ are arbitrary constants. Moreover, by
the formula (\ref{d1}) $\beta_{1}$ and $\beta_{2}$ order derivatives are as
follows:
\begin{align*}
\mathfrak{D}_{0^{+}}^{\beta_{1}}u(t)  &  =I^{\alpha-\beta_{1}}f(t)-k_{1}%
\frac{t^{1-\beta_{1}}}{\Gamma(2-\beta_{1})}-2k_{2}\frac{t^{2-\beta_{1}}%
}{\Gamma(3-\beta_{1})},\\
\mathfrak{D}_{0^{+}}^{\beta_{2}}u(t)  &  =I^{\alpha-\beta_{2}}f(t)-2k_{2}%
\frac{t^{2-\beta_{2}}}{\Gamma(3-\beta_{2})}.
\end{align*}
Using boundary conditions (\ref{e2}), we get the following algebraic system of
equations for $k_{0},k_{1},k_{2}$.%
\begin{align*}
-\left(  a_{0}+b_{0}\right)  k_{0}-b_{0}Tk_{1}-b_{0}T^{2}k_{2}  &
=\lambda_{0}\int\limits_{0}^{T}g_{0}(s)ds-b_{0}I_{0^{+}}^{\alpha}f(T),\\
-\frac{a_{1}\eta^{1-\beta_{1}}+b_{1}T^{1-\beta_{1}}}{\Gamma(2-\beta_{1})}%
k_{1}-2\frac{a_{1}\eta^{2-\beta_{1}}+b_{1}T^{2-\beta_{1}}}{\Gamma(3-\beta
_{1})}k_{2}  &  =\lambda_{1}\int\limits_{0}^{T}g_{1}(s)ds-a_{1}I_{0^{+}%
}^{\alpha-\beta_{1}}f(\eta)-b_{1}I_{0^{+}}^{\alpha-\beta_{1}}f(T),\\
-2\frac{a_{2}\eta^{2-\beta_{2}}+b_{2}T^{2-\beta_{2}}}{\Gamma(3-\beta_{2}%
)}k_{2}  &  =\lambda_{2}\int\limits_{0}^{T}g_{2}(s)ds-a_{2}I_{0^{+}}%
^{\alpha-\beta_{2}}f(\eta)-b_{2}I_{0^{+}}^{\alpha-\beta_{2}}f(T).
\end{align*}
Solving the above system of equations for $k_{0},k_{1},k_{2}$, we get the
following:%
\begin{align*}
k_{2}  &  =b_{2}\mu^{\beta_{2}}I_{0^{+}}^{\alpha-\beta_{2}}f(T)+a_{2}%
\mu^{\beta_{2}}I_{0^{+}}^{\alpha-\beta_{2}}f(\eta)-\lambda_{2}\mu^{\beta_{2}%
}\int\limits_{0}^{T}g_{2}(s)ds,\\
k_{1}  &  =b_{1}\nu^{\beta_{1}}I_{0^{+}}^{\alpha-\beta_{1}}f(T)+a_{1}%
\nu^{\beta_{1}}I_{0^{+}}^{\alpha-\beta_{1}}f(\eta)-\lambda_{1}\nu^{\beta_{1}%
}\int\limits_{0}^{T}g_{1}(s)ds\\
&  -b_{2}\nu^{\beta_{1}}\dfrac{\mu^{\beta_{2}}}{\mu^{\beta_{1}}}I_{0^{+}%
}^{\alpha-\beta_{2}}f(T)-a_{2}\nu^{\beta_{1}}\dfrac{\mu^{\beta_{2}}}%
{\mu^{\beta_{1}}}I_{0^{+}}^{\alpha-\beta_{2}}f(\eta)+\lambda_{2}\nu^{\beta
_{1}}\dfrac{\mu^{\beta_{2}}}{\mu^{\beta_{1}}}\int\limits_{0}^{T}g_{2}(s)ds,\\
k_{0}  &  =\dfrac{b_{0}}{a_{0}+b_{0}}I_{0^{+}}^{\alpha}f(T)-\dfrac{\lambda
_{0}}{a_{0}+b_{0}}\int\limits_{0}^{T}g_{0}(s)ds\\
&  -\dfrac{b_{0}b_{1}\nu^{\beta_{1}}T}{a_{0}+b_{0}}I_{0^{+}}^{\alpha-\beta
_{1}}f(T)-\dfrac{b_{0}a_{1}\nu^{\beta_{1}}T}{a_{0}+b_{0}}I_{0^{+}}%
^{\alpha-\beta_{1}}f(\eta)+\dfrac{b_{0}\lambda_{1}\nu^{\beta_{1}}T}%
{a_{0}+b_{0}}\int\limits_{0}^{T}g_{1}(s)ds\\
&  +\dfrac{b_{0}b_{2}\nu^{\beta_{1}}T}{a_{0}+b_{0}}\dfrac{\mu^{\beta_{2}}}%
{\mu^{\beta_{1}}}I_{0^{+}}^{\alpha-\beta_{2}}f(T)+\dfrac{b_{0}a_{2}\nu
^{\beta_{1}}T}{a_{0}+b_{0}}\dfrac{\mu^{\beta_{2}}}{\mu^{\beta_{1}}}I_{0^{+}%
}^{\alpha-\beta_{2}}f(\eta)-\dfrac{b_{0}\lambda_{2}\nu^{\beta_{1}}T}%
{a_{0}+b_{0}}\dfrac{\mu^{\beta_{2}}}{\mu^{\beta_{1}}}\int\limits_{0}^{T}%
g_{2}(s)ds\\
&  -\dfrac{b_{0}b_{2}\mu^{\beta_{2}}T^{2}}{a_{0}+b_{0}}I_{0^{+}}^{\alpha
-\beta_{2}}f(T)-\dfrac{b_{0}a_{2}\mu^{\beta_{2}}T^{2}}{a_{0}+b_{0}}I_{0^{+}%
}^{\alpha-\beta_{2}}f(\eta)+\dfrac{b_{0}\lambda_{2}\mu^{\beta_{2}}T^{2}}%
{a_{0}+b_{0}}\int\limits_{0}^{T}g_{2}(s)ds.
\end{align*}
Inserting $k_{0},k_{1},k_{2}$ into (\ref{ss1}) we get the desired
representation for the solution of (\ref{e1})-(\ref{e2}).
\end{proof}

\begin{remark}
\label{Rem:1}The Green function of the BVP is defined by
\[
G(t;s)=\left\{
\begin{array}
[c]{c}%
-\dfrac{(t-s)^{\alpha-1}}{\Gamma(\alpha)}+G_{0}(t;s),\text{ \ \ }0\leq s\leq
t\leq T,\\
G_{0}(t;s),\text{ \ \ \ \ }0\leq t\leq s\leq T,
\end{array}
\right.
\]
where
\begin{align*}
G_{0}(t;s)  &  =%
{\displaystyle\sum\limits_{i=0}^{2}}
\omega_{i}\left(  t\right)  b_{i}\dfrac{(T-s)^{\alpha-\beta_{i}-1}}%
{\Gamma(\alpha-\beta_{i})}+%
{\displaystyle\sum\limits_{i=1}^{2}}
\omega_{i}\left(  t\right)  a_{i}\dfrac{(\eta-s)^{\alpha-\beta_{i}-1}}%
{\Gamma(\alpha-\beta_{i})}\chi_{\left(  0,\eta\right)  }\left(  s\right)  ,\\
\chi_{\left(  a,b\right)  }\left(  s\right)   &  :=\left\{
\begin{array}
[c]{c}%
1,\ \ \ s\in\left(  a,b\right)  ,\\
0,\ \ \ s\notin\left(  a,b\right)  .
\end{array}
\right.  .
\end{align*}

\end{remark}

\begin{remark}
For $\alpha=3,\beta_{1}=1,\beta_{2}=2$ and $\eta=0,$ the Green function of
(\ref{e1})-(\ref{e2}) can be writen as follows:%
\[
G(t;s)=\left\{
\begin{array}
[c]{c}%
-\dfrac{(t-s)^{\alpha-1}}{\Gamma(\alpha)}+G_{0}(t;s),\text{ \ \ }0\leq s\leq
t\leq T,\\
G_{0}(t;s),\text{ \ \ \ \ }0\leq t\leq s\leq T.
\end{array}
\right.
\]
where%
\begin{align*}
&  G_{0}(t;s)=\dfrac{b_{0}}{a_{0}+b_{0}}\dfrac{(T-s)^{\alpha-1}}{\Gamma
(\alpha)}+\left(  -\frac{b_{0}T}{a_{0}+b_{0}}\frac{b_{1}}{a_{1}+b_{1}}%
+\frac{b_{1}}{a_{1}+b_{1}}t\right)  \dfrac{(T-s)^{\alpha-2}}{\Gamma(\alpha
-1)}\\
&  +\left(  \frac{b_{0}}{a_{0}+b_{0}}\frac{b_{1}}{a_{1}+b_{1}}\frac{b_{2}%
}{a_{2}+b_{2}}T-\frac{b_{0}T^{2}}{a_{0}+b_{0}}\frac{b_{2}}{2\left(
a_{2}+b_{2}\right)  }-\frac{2b_{1}}{a_{1}+b_{1}}\frac{b_{2}}{2\left(
a_{2}+b_{2}\right)  }t+\frac{b_{2}}{2\left(  a_{2}+b_{2}\right)  }%
t^{2}\right)  \dfrac{(T-s)^{\alpha-3}}{\Gamma(\alpha-2)}%
\end{align*}
Moreover, the case%
\[
a_{0}=1,b_{0}=0,a_{1}=0,b_{1}=1,a_{2}=1,b_{2}=0
\]
is investigated in \cite{ait}. In this case,%
\[
G(t;s)=\left\{
\begin{array}
[c]{c}%
-\dfrac{(t-s)^{\alpha-1}}{\Gamma(\alpha)}+t\dfrac{(T-s)^{\alpha-2}}%
{\Gamma(\alpha-1)},\text{ \ \ }0\leq s\leq t\leq T,\\
t\dfrac{(T-s)^{\alpha-2}}{\Gamma(\alpha-1)},\text{ \ \ \ \ }0\leq t\leq s\leq
T.
\end{array}
\right.
\]

\end{remark}

\section{Existence and uniqueness results}

In this section we state and prove an existence and uniqueness result for the
fractional BVP (\ref{p1})-(\ref{p2}) by using the Banach fixed-point theorem.
We study our problem in the space%
\[
C_{\beta}\left(  \left[  0,T\right]  ;\mathbb{R}\right)  :=\left\{  v\in
C\left(  \left[  0,T\right]  ;\mathbb{R}\right)  :\mathfrak{D}_{0^{+}}%
^{\beta_{1}}v,\ \mathfrak{D}_{0^{+}}^{\beta_{2}}v\in C\left(  \left[
0,T\right]  ;\mathbb{R}\right)  \right\}
\]
equipped with the norm%
\[
\left\Vert v\right\Vert _{\beta}:=\left\Vert v\right\Vert _{C}+\left\Vert
\mathfrak{D}_{0^{+}}^{\beta_{1}}v\right\Vert _{C}+\left\Vert \mathfrak{D}%
_{0^{+}}^{\beta_{2}}v\right\Vert _{C},
\]
where $\left\Vert \cdot\right\Vert _{C}$ is the sup norm in $C\left(  \left[
0,T\right]  ;\mathbb{R}\right)  $.

The following notations, formulae and estimations will be used throughout the
paper.%
\begin{align*}
\mathfrak{D}_{0^{+}}^{\beta_{1}}\omega_{1}\left(  t\right)   &  =-\dfrac
{\nu^{\beta_{1}}t^{1-\beta_{1}}}{\Gamma(2-\beta_{1})},...\mathfrak{D}_{0^{+}%
}^{\beta_{1}}\omega_{1}\left(  t\right)  =0,\\
\mathfrak{D}_{0^{+}}^{\beta_{1}}\omega_{2}\left(  t\right)   &  =\dfrac
{\nu^{\beta_{1}}\mu^{\beta_{2}}t^{1-\beta_{1}}}{\mu^{\beta_{1}}\Gamma
(2-\beta_{1})}-2\dfrac{\mu^{\beta_{2}}t^{2-\beta_{1}}}{\Gamma(3-\beta_{1}%
)},\ \ \mathfrak{D}_{0^{+}}^{\beta_{2}}\omega_{2}\left(  t\right)
=-2\dfrac{\mu^{\beta_{2}}t^{2-\beta_{2}}}{\Gamma(3-\beta_{2})}.
\end{align*}%
\begin{align*}
\left\vert \omega_{0}\right\vert  &  =\dfrac{1}{\left\vert a_{0}%
+b_{0}\right\vert }=:\rho_{0},\ \ \ \left\vert \omega_{1}\left(  t\right)
\right\vert \leq\left\vert \nu^{\beta_{1}}\right\vert \left(  \left\vert
\omega_{0}\right\vert \left\vert b_{0}\right\vert +1\right)  T:=\rho
_{1},\ \ \ \\
\left\vert \omega_{2}\left(  t\right)  \right\vert  &  \leq\dfrac{\left\vert
b_{0}\right\vert \left\vert \mu^{\beta_{2}}\right\vert }{\left\vert
a_{0}+b_{0}\right\vert }T^{2}+\dfrac{\left\vert b_{0}\right\vert \left\vert
\nu^{\beta_{1}}\right\vert }{\left\vert a_{0}+b_{0}\right\vert }%
\dfrac{\left\vert \mu^{\beta_{2}}\right\vert }{\left\vert \mu^{\beta_{1}%
}\right\vert }T+\dfrac{\left\vert \nu^{\beta_{1}}\right\vert \left\vert
\mu^{\beta_{2}}\right\vert }{\left\vert \mu^{\beta_{1}}\right\vert
}T+\left\vert \mu^{\beta_{2}}\right\vert T^{2}:=\rho_{2}.\\
\widetilde{\rho}_{0}  &  =0,\ \ \left\vert \mathfrak{D}_{0^{+}}^{\beta_{1}%
}\omega_{1}\left(  t\right)  \right\vert \leq\dfrac{\left\vert \nu^{\beta_{1}%
}\right\vert T^{1-\beta_{1}}}{\Gamma(2-\beta_{1})}:=\widetilde{\rho}%
_{1},\ \ \ \left\vert \mathfrak{D}_{0^{+}}^{\beta_{1}}\omega_{2}\left(
t\right)  \right\vert \leq\dfrac{\left\vert \mu^{\beta_{2}}\right\vert
\left\vert \nu^{\beta_{1}}\right\vert T^{1-\beta_{1}}}{\left\vert \mu
^{\beta_{1}}\right\vert \Gamma(2-\beta_{1})}+2\dfrac{\left\vert \mu^{\beta
_{2}}\right\vert T^{2-\beta_{1}}}{\Gamma(3-\beta_{1})}:=\widetilde{\rho}%
_{2},\\
\widehat{\rho}_{0}  &  =\widehat{\rho}_{1}=0,\ \ \ \left\vert \mathfrak{D}%
_{0^{+}}^{\beta_{2}}\omega_{2}\left(  t\right)  \right\vert \leq
2\dfrac{\left\vert \mu^{\beta_{2}}\right\vert T^{2-\beta_{2}}}{\Gamma
(3-\beta_{1})}:=\widehat{\rho}_{2}.
\end{align*}%
\begin{gather*}
\Delta_{0}:=\dfrac{T^{\alpha-\tau}}{\Gamma(\alpha)}\left(  \frac{1-\tau
}{\alpha-\tau}\right)  ^{1-\tau}+%
{\displaystyle\sum\limits_{i=0}^{2}}
\rho_{i}\left(  \left\vert b_{i}\right\vert \dfrac{T^{\alpha-\beta_{i}-\tau}%
}{\Gamma(\alpha-\beta_{i})}+\left\vert a_{i}\right\vert \dfrac{\eta
^{\alpha-\beta_{i}-\tau}}{\Gamma(\alpha-\beta_{i})}\right)  \left(
\frac{1-\tau}{\alpha-\beta_{i}-\tau}\right)  ^{1-\tau},\\
\Delta_{1}:=\dfrac{l_{f}T^{\alpha-\beta_{1}}}{\Gamma(\alpha-\beta_{1}+1)}+%
{\displaystyle\sum\limits_{i=1}^{2}}
\widetilde{\rho}_{i}\left(  \left\vert b_{i}\right\vert \dfrac{l_{f}%
T^{\alpha-\beta_{i}}}{\Gamma(\alpha-\beta_{i}+1)}+\left\vert a_{i}\right\vert
\dfrac{l_{f}\eta^{\alpha-\beta_{i}}}{\Gamma(\alpha-\beta_{i}+1)}\right)  ,\\
\Delta_{2}:=\dfrac{l_{f}T^{\alpha-\beta_{2}}}{\Gamma(\alpha-\beta_{2}%
+1)}+\widehat{\rho}_{2}\left(  \left\vert b_{2}\right\vert \dfrac
{l_{f}T^{\alpha-\beta_{2}}}{\Gamma(\alpha-\beta_{2}+1)}+\left\vert
a_{2}\right\vert \dfrac{l_{f}\eta^{\alpha-\beta_{2}}}{\Gamma(\alpha-\beta
_{2}+1)}\right)  .
\end{gather*}

\begin{theorem}
\label{Thm:uniq}Assume that

\begin{enumerate}
\item[(H$_{1}$)] The function $f:\left[  0,T\right]  \times\mathbb{R}%
\times\mathbb{R}\times\mathbb{R}\rightarrow\mathbb{R}$ is jointly continuous.

\item[(H$_{2}$)] There exists a function $l_{f}\in L^{\frac{1}{\tau}}\left(
\left[  0,T\right]  ;\mathbb{R}^{+}\right)  $ with $\tau\in\left(
0,\alpha-\beta_{2}\right)  $ such that%
\[
\left\vert f\left(  t,u_{1},u_{2},u_{3}\right)  -f\left(  t,v_{1},v_{2}%
,v_{3}\right)  \right\vert \leq l_{f}\left(  t\right)  \left(  \left\vert
u_{1}-v_{1}\right\vert +\left\vert u_{2}-v_{2}\right\vert +\left\vert
u_{3}-v_{3}\right\vert \right)  ,
\]
for each $\left(  t,u_{1},u_{2},u_{3}\right)  ,\left(  t,v_{1},v_{2}%
,v_{3}\right)  \in\left[  0,T\right]  \times\mathbb{R}\times\mathbb{R}%
\times\mathbb{R}.$

\item[(H$_{3}$)] The function $g_{i}:\left[  0,T\right]  \times\mathbb{R}%
\rightarrow\mathbb{R}$ is jointly continuous and there exists $l_{g_{i}}\in
L^{1}\left(  \left[  0,T\right]  ,\mathbb{R}^{+}\right)  $ such that%
\[
\left\vert g_{i}\left(  t,u\right)  -g_{i}\left(  t,v\right)  \right\vert \leq
l_{g_{i}}\left(  t\right)  \left\vert u-v\right\vert ,\ i=0,1,2
\]
for each $\left(  t,u\right)  ,\left(  t,v\right)  \in\left[  0,T\right]
\times\mathbb{R}.$\newline If%
\begin{equation}
\left(  \Delta_{0}+\Delta_{1}+\Delta_{2}\right)  \left\Vert l_{f}\right\Vert
_{1/\tau}+%
{\displaystyle\sum\limits_{i=0}^{2}}
\rho_{i}\left\vert \lambda_{i}\right\vert \left\Vert l_{g_{i}}\right\Vert
_{1}+%
{\displaystyle\sum\limits_{i=1}^{2}}
\widetilde{\rho}_{i}\left\vert \lambda_{i}\right\vert \left\Vert l_{g_{i}%
}\right\Vert _{1}+\widehat{\rho}_{2}\left\vert \lambda_{2}\right\vert
\left\Vert l_{g_{2}}\right\Vert _{1}<1, \label{cc1}%
\end{equation}
then the problem (\ref{p1})-(\ref{p2}) has a unique solution on $\left[
0,T\right]  $.
\end{enumerate}
\end{theorem}

\begin{proof}
In order to transform the BVP (\ref{p1})-(\ref{p2}) into a fixed point
problem, we consider the operator $\mathfrak{F}:C_{\beta}\left(  \left[
0,T\right]  ;\mathbb{R}\right)  \rightarrow C_{\beta}\left(  \left[
0,T\right]  ;\mathbb{R}\right)  $ which is defined by%
\begin{align}
\left(  \mathfrak{F}u\right)  \left(  t\right)   &  =%
{\displaystyle\int\limits_{0}^{t}}
\dfrac{(t-s)^{\alpha-1}}{\Gamma(\alpha)}f(s,u\left(  s\right)  ,\mathfrak{D}%
_{0^{+}}^{\beta_{1}}u(s),\mathfrak{D}_{0^{+}}^{\beta_{2}}u(s))ds\nonumber\\
&  +%
{\displaystyle\sum\limits_{i=0}^{2}}
\omega_{i}\left(  t\right)  b_{i}%
{\displaystyle\int\limits_{0}^{T}}
\frac{(T-s)^{\alpha-\beta_{i}-1}}{\Gamma(\alpha-\beta_{i})}f(s,u\left(
s\right)  ,\mathfrak{D}_{0^{+}}^{\beta_{1}}u(s),\mathfrak{D}_{0^{+}}%
^{\beta_{2}}u(s))ds\nonumber\\
&  +%
{\displaystyle\sum\limits_{i=1}^{2}}
\omega_{i}\left(  t\right)  a_{i}%
{\displaystyle\int\limits_{0}^{\eta}}
\frac{(\eta-s)^{\alpha-\beta_{i}-1}}{\Gamma(\alpha-\beta_{i})}f(s,u\left(
s\right)  ,\mathfrak{D}_{0^{+}}^{\beta_{1}}u(s),\mathfrak{D}_{0^{+}}%
^{\beta_{2}}u(s))ds-%
{\displaystyle\sum\limits_{i=0}^{2}}
\omega_{i}\left(  t\right)  \lambda_{i}%
{\displaystyle\int\limits_{0}^{T}}
g_{i}(s,u\left(  s\right)  )ds, \label{ff1}%
\end{align}
and take its $\beta_{1}$-th and $\beta_{2}$-th fractional derivative to get%
\begin{align}
\mathfrak{D}_{0^{+}}^{\beta_{1}}\left(  \mathfrak{F}u\right)  \left(
t\right)   &  =%
{\displaystyle\int\limits_{0}^{t}}
\dfrac{(t-s)^{\alpha-\beta_{1}-1}}{\Gamma(\alpha-\beta_{1})}f(s,u\left(
s\right)  ,\mathfrak{D}_{0^{+}}^{\beta_{1}}u(s),\mathfrak{D}_{0^{+}}%
^{\beta_{2}}u(s))ds\nonumber\\
&  +%
{\displaystyle\sum\limits_{i=1}^{2}}
\mathfrak{D}_{0^{+}}^{\beta_{1}}\omega_{i}\left(  t\right)  b_{i}%
{\displaystyle\int\limits_{0}^{T}}
\frac{(T-s)^{\alpha-\beta_{i}-1}}{\Gamma(\alpha-\beta_{i})}f(s,u\left(
s\right)  ,\mathfrak{D}_{0^{+}}^{\beta_{1}}u(s),\mathfrak{D}_{0^{+}}%
^{\beta_{2}}u(s))ds\nonumber\\
&  +%
{\displaystyle\sum\limits_{i=1}^{2}}
\mathfrak{D}_{0^{+}}^{\beta_{1}}\omega_{i}\left(  t\right)  a_{i}%
{\displaystyle\int\limits_{0}^{\eta}}
\frac{(\eta-s)^{\alpha-\beta_{i}-1}}{\Gamma(\alpha-\beta_{i})}f(s,u\left(
s\right)  ,\mathfrak{D}_{0^{+}}^{\beta_{1}}u(s),\mathfrak{D}_{0^{+}}%
^{\beta_{2}}u(s))ds-%
{\displaystyle\sum\limits_{i=1}^{2}}
\mathfrak{D}_{0^{+}}^{\beta_{1}}\omega_{i}\left(  t\right)  \lambda_{i}%
{\displaystyle\int\limits_{0}^{T}}
g_{i}(s,u\left(  s\right)  )ds, \label{ff2}%
\end{align}
and%
\begin{align}
\mathfrak{D}_{0^{+}}^{\beta_{2}}\left(  \mathfrak{F}u\right)  \left(
t\right)   &  =%
{\displaystyle\int\limits_{0}^{t}}
\dfrac{(t-s)^{\alpha-\beta_{2}-1}}{\Gamma(\alpha-\beta_{2})}f(s,u\left(
s\right)  ,\mathfrak{D}_{0^{+}}^{\beta_{1}}u(s),\mathfrak{D}_{0^{+}}%
^{\beta_{2}}u(s))ds\nonumber\\
&  +\mathfrak{D}_{0^{+}}^{\beta_{2}}\omega_{2}\left(  t\right)  b_{2}%
{\displaystyle\int\limits_{0}^{T}}
\frac{(T-s)^{\alpha-\beta_{2}-1}}{\Gamma(\alpha-\beta_{2})}f(s,u\left(
s\right)  ,\mathfrak{D}_{0^{+}}^{\beta_{1}}u(s),\mathfrak{D}_{0^{+}}%
^{\beta_{2}}u(s))ds\nonumber\\
&  +\mathfrak{D}_{0^{+}}^{\beta_{2}}\omega_{2}\left(  t\right)  a_{2}%
{\displaystyle\int\limits_{0}^{\eta}}
\frac{(\eta-s)^{\alpha-\beta_{2}-1}}{\Gamma(\alpha-\beta_{2})}f(s,u\left(
s\right)  ,\mathfrak{D}_{0^{+}}^{\beta_{1}}u(s),\mathfrak{D}_{0^{+}}%
^{\beta_{2}}u(s))ds-\mathfrak{D}_{0^{+}}^{\beta_{2}}\omega_{2}\left(
t\right)  \lambda_{2}%
{\displaystyle\int\limits_{0}^{T}}
g_{2}(s,u\left(  s\right)  )ds. \label{ff3}%
\end{align}
Clearly, due to $f,g_{0},g_{1},g_{2}$ being jointly continuous, the
expressions (\ref{ff1})-(\ref{ff3}) are well defined. It is obvious that the
fixed point of the operator $\mathfrak{F}$ is a solution of the problem
(\ref{p1})-(\ref{p2}). To show existence and uniquiness of the solution
(\ref{e1})-(\ref{e2}) we use the Banach fixed point theorem. To this end, we
show that $\mathfrak{F}$ is contraction.%
\begin{align}
&  \left\vert \left(  \mathfrak{F}u\right)  \left(  t\right)  -\left(
\mathfrak{F}v\right)  \left(  t\right)  \right\vert \nonumber\\
&  \leq%
{\displaystyle\int\limits_{0}^{t}}
\dfrac{(t-s)^{\alpha-1}}{\Gamma(\alpha)}\left\vert f(s,u\left(  s\right)
,\mathfrak{D}_{0^{+}}^{\beta_{1}}u(s),\mathfrak{D}_{0^{+}}^{\beta_{2}%
}u(s))-f(s,v\left(  s\right)  ,\mathfrak{D}_{0^{+}}^{\beta_{1}}%
v(s),\mathfrak{D}_{0^{+}}^{\beta_{2}}v(s))\right\vert ds\nonumber\\
&  +%
{\displaystyle\sum\limits_{i=0}^{2}}
\left\vert \omega_{i}\left(  t\right)  \right\vert \left\vert b_{i}%
\right\vert
{\displaystyle\int\limits_{0}^{T}}
\frac{(T-s)^{\alpha-\beta_{i}-1}}{\Gamma(\alpha-\beta_{i})}\left\vert
f(s,u\left(  s\right)  ,\mathfrak{D}_{0^{+}}^{\beta_{1}}u(s),\mathfrak{D}%
_{0^{+}}^{\beta_{2}}u(s))-f(s,v\left(  s\right)  ,\mathfrak{D}_{0^{+}}%
^{\beta_{1}}v(s),\mathfrak{D}_{0^{+}}^{\beta_{2}}v(s))\right\vert
ds\nonumber\\
&  +%
{\displaystyle\sum\limits_{i=1}^{2}}
\left\vert \omega_{i}\left(  t\right)  \right\vert \left\vert a_{i}%
\right\vert
{\displaystyle\int\limits_{0}^{\eta}}
\frac{(\eta-s)^{\alpha-\beta_{i}-1}}{\Gamma(\alpha-\beta_{i})}\left\vert
f(s,u\left(  s\right)  ,\mathfrak{D}_{0^{+}}^{\beta_{1}}u(s),\mathfrak{D}%
_{0^{+}}^{\beta_{2}}u(s))-f(s,v\left(  s\right)  ,\mathfrak{D}_{0^{+}}%
^{\beta_{1}}v(s),\mathfrak{D}_{0^{+}}^{\beta_{2}}v(s))\right\vert
ds\nonumber\\
&  +%
{\displaystyle\sum\limits_{i=0}^{2}}
\left\vert \omega_{i}\left(  t\right)  \right\vert \left\vert \lambda
_{i}\right\vert
{\displaystyle\int\limits_{0}^{T}}
\left\vert g_{i}(s,u\left(  s\right)  )-g_{i}(s,v\left(  s\right)
)\right\vert ds\nonumber\\
&  \leq\left\Vert l_{f}\right\Vert _{1/\tau}\dfrac{T^{\alpha-\tau}}%
{\Gamma(\alpha)}\left(  \frac{1-\tau}{\alpha-\tau}\right)  ^{1-\tau}\left\Vert
u-v\right\Vert _{\beta}\nonumber\\
&  +\left\Vert l_{f}\right\Vert _{1/\tau}%
{\displaystyle\sum\limits_{i=0}^{2}}
\rho_{i}\left(  \left\vert b_{i}\right\vert \dfrac{T^{\alpha-\beta_{i}-\tau}%
}{\Gamma(\alpha-\beta_{i})}+\left\vert a_{i}\right\vert \dfrac{\eta
^{\alpha-\beta_{i}-\tau}}{\Gamma(\alpha-\beta_{i})}\right)  \left(
\frac{1-\tau}{\alpha-\beta_{i}-\tau}\right)  ^{1-\tau}+%
{\displaystyle\sum\limits_{i=0}^{2}}
\rho_{i}\left\vert \lambda_{i}\right\vert \left\Vert l_{g_{i}}\right\Vert
_{1}\left\Vert u-v\right\Vert _{\beta}\nonumber\\
&  =\left(  \Delta_{0}\left\Vert l_{f}\right\Vert _{1/\tau}+%
{\displaystyle\sum\limits_{i=0}^{2}}
\rho_{i}\left\vert \lambda_{i}\right\vert \left\Vert l_{g_{i}}\right\Vert
_{1}\right)  \left\Vert u-v\right\Vert _{\beta}. \label{f1}%
\end{align}
On the other hand,
\begin{align}
&  \left\vert \mathfrak{D}_{0+}^{\beta_{1}}\left(  \mathfrak{F}u\right)
\left(  t\right)  -\mathfrak{D}_{0+}^{\beta_{1}}\left(  \mathfrak{F}v\right)
\left(  t\right)  \right\vert \nonumber\\
&  \leq%
{\displaystyle\int\limits_{0}^{t}}
\dfrac{(t-s)^{\alpha-\beta_{1}-1}}{\Gamma(\alpha-\beta_{1})}\left\vert
f(s,u\left(  s\right)  ,\mathfrak{D}_{0^{+}}^{\beta_{1}}u(s),\mathfrak{D}%
_{0^{+}}^{\beta_{2}}u(s))-f(s,v\left(  s\right)  ,\mathfrak{D}_{0^{+}}%
^{\beta_{1}}v(s),\mathfrak{D}_{0^{+}}^{\beta_{2}}v(s))\right\vert
ds\nonumber\\
&  +%
{\displaystyle\sum\limits_{i=1}^{2}}
\left\vert \mathfrak{D}_{0+}^{\beta_{1}}\omega_{i}\left(  t\right)
\right\vert \left\vert b_{i}\right\vert
{\displaystyle\int\limits_{0}^{T}}
\frac{(T-s)^{\alpha-\beta_{i}-1}}{\Gamma(\alpha-\beta_{i})}\left\vert
f(s,u\left(  s\right)  ,\mathfrak{D}_{0^{+}}^{\beta_{1}}u(s),\mathfrak{D}%
_{0^{+}}^{\beta_{2}}u(s))-f(s,v\left(  s\right)  ,\mathfrak{D}_{0^{+}}%
^{\beta_{1}}v(s),\mathfrak{D}_{0^{+}}^{\beta_{2}}v(s))\right\vert
ds\nonumber\\
&  +%
{\displaystyle\sum\limits_{i=1}^{2}}
\left\vert \mathfrak{D}_{0+}^{\beta_{1}}\omega_{i}\left(  t\right)
\right\vert \left\vert a_{i}\right\vert
{\displaystyle\int\limits_{0}^{\eta}}
\frac{(\eta-s)^{\alpha-\beta_{i}-1}}{\Gamma(\alpha-\beta_{i})}\left\vert
f(s,u\left(  s\right)  ,\mathfrak{D}_{0^{+}}^{\beta_{1}}u(s),\mathfrak{D}%
_{0^{+}}^{\beta_{2}}u(s))-f(s,v\left(  s\right)  ,\mathfrak{D}_{0^{+}}%
^{\beta_{1}}v(s),\mathfrak{D}_{0^{+}}^{\beta_{2}}v(s))\right\vert
ds\nonumber\\
&  +%
{\displaystyle\sum\limits_{i=1}^{2}}
\left\vert \mathfrak{D}_{0+}^{\beta_{1}}\omega_{i}\left(  t\right)
\right\vert \left\vert \lambda_{i}\right\vert
{\displaystyle\int\limits_{0}^{T}}
\left\vert g_{i}(s,u\left(  s\right)  )-g_{i}(s,v\left(  s\right)
)\right\vert ds\nonumber\\
&  \leq\dfrac{T^{\alpha-\beta_{1}-\tau}}{\Gamma(\alpha-\beta_{1})}\left(
\frac{1-\tau}{\alpha-\beta_{1}-\tau}\right)  ^{1-\tau}\left\Vert
l_{f}\right\Vert _{1/\tau}\left\Vert u-v\right\Vert _{\beta}\nonumber\\
&  +%
{\displaystyle\sum\limits_{i=1}^{2}}
\widetilde{\rho}_{i}\left[  \left\Vert l_{f}\right\Vert _{1/\tau}\left(
\left\vert b_{i}\right\vert \dfrac{T^{\alpha-\beta_{i}-\tau}}{\Gamma
(\alpha-\beta_{i})}+\left\vert a_{i}\right\vert \dfrac{\eta^{\alpha-\beta
_{i}-\tau}}{\Gamma(\alpha-\beta_{i})}\right)  \left(  \frac{1-\tau}%
{\alpha-\beta_{i}-\tau}\right)  ^{1-\tau}+\left\vert \lambda_{i}\right\vert
\left\Vert l_{g_{i}}\right\Vert _{1}\right]  \left\Vert u-v\right\Vert
_{\beta}\nonumber\\
&  =\left(  \Delta_{1}\left\Vert l_{f}\right\Vert _{1/\tau}+%
{\displaystyle\sum\limits_{i=1}^{2}}
\widetilde{\rho}_{i}\left\vert \lambda_{i}\right\vert \left\Vert l_{g_{i}%
}\right\Vert _{1}\right)  \left\Vert u-v\right\Vert _{\beta}. \label{f11}%
\end{align}
Similarly%
\begin{align}
&  \left\vert \mathfrak{D}_{0+}^{\beta_{2}}\left(  \mathfrak{F}u\right)
\left(  t\right)  -\mathfrak{D}_{0+}^{\beta_{2}}\left(  \mathfrak{F}v\right)
\left(  t\right)  \right\vert \nonumber\\
&  \leq%
{\displaystyle\int\limits_{0}^{t}}
\dfrac{(t-s)^{\alpha-\beta_{2}-1}}{\Gamma(\alpha-\beta_{2})}\left\vert
f(s,u\left(  s\right)  ,\mathfrak{D}_{0^{+}}^{\beta_{1}}u(s),\mathfrak{D}%
_{0^{+}}^{\beta_{2}}u(s))-f(s,v\left(  s\right)  ,\mathfrak{D}_{0^{+}}%
^{\beta_{1}}v(s),\mathfrak{D}_{0^{+}}^{\beta_{2}}v(s))\right\vert
ds\nonumber\\
&  +\left\vert \mathfrak{D}_{0+}^{\beta_{2}}\omega_{2}\left(  t\right)
\right\vert \left\vert b_{2}\right\vert
{\displaystyle\int\limits_{0}^{T}}
\frac{(T-s)^{\alpha-\beta_{2}-1}}{\Gamma(\alpha-\beta_{2})}\left\vert
f(s,u\left(  s\right)  ,\mathfrak{D}_{0^{+}}^{\beta_{1}}u(s),\mathfrak{D}%
_{0^{+}}^{\beta_{2}}u(s))-f(s,v\left(  s\right)  ,\mathfrak{D}_{0^{+}}%
^{\beta_{1}}v(s),\mathfrak{D}_{0^{+}}^{\beta_{2}}v(s))\right\vert
ds\nonumber\\
&  +\left\vert \mathfrak{D}_{0+}^{\beta_{2}}\omega_{2}\left(  t\right)
\right\vert \left\vert a_{2}\right\vert
{\displaystyle\int\limits_{0}^{\eta}}
\frac{(\eta-s)^{\alpha-\beta_{2}-1}}{\Gamma(\alpha-\beta_{2})}\left\vert
f(s,u\left(  s\right)  ,\mathfrak{D}_{0^{+}}^{\beta_{1}}u(s),\mathfrak{D}%
_{0^{+}}^{\beta_{2}}u(s))-f(s,v\left(  s\right)  ,\mathfrak{D}_{0^{+}}%
^{\beta_{1}}v(s),\mathfrak{D}_{0^{+}}^{\beta_{2}}v(s))\right\vert
ds\nonumber\\
&  +\left\vert \mathfrak{D}_{0+}^{\beta_{2}}\omega_{2}\left(  t\right)
\right\vert \left\vert \lambda_{2}\right\vert
{\displaystyle\int\limits_{0}^{T}}
\left\vert g_{2}(s,u\left(  s\right)  )-g_{2}(s,v\left(  s\right)
)\right\vert ds\nonumber\\
&  \leq\dfrac{T^{\alpha-\beta_{2}-\tau}}{\Gamma(\alpha-\beta_{2})}\left(
\frac{1-\tau}{\alpha-\beta_{2}-\tau}\right)  ^{1-\tau}\left\Vert
l_{f}\right\Vert _{1/\tau}\left\Vert u-v\right\Vert _{\beta}\nonumber\\
&  +\widehat{\rho}_{2}\left(  \left\vert b_{2}\right\vert \frac{T^{\alpha
-\beta_{2}-\tau}}{\Gamma(\alpha-\beta_{2})}+\left\vert a_{2}\right\vert
\frac{\eta^{\alpha-\beta_{2}-\tau}}{\Gamma(\alpha-\beta_{2})}\right)  \left(
\frac{1-\tau}{\alpha-\beta_{2}-\tau}\right)  ^{1-\tau}\left\Vert
l_{f}\right\Vert _{1/\tau}+\widehat{\rho}_{2}\left\vert \lambda_{2}\right\vert
\left\Vert l_{g_{2}}\right\Vert _{1}\left\Vert u-v\right\Vert _{\beta
}\nonumber\\
&  =\left(  \Delta_{2}\left\Vert l_{f}\right\Vert _{1/\tau}+\widehat{\rho}%
_{2}\left\vert \lambda_{2}\right\vert \left\Vert l_{g_{2}}\right\Vert
_{1}\right)  \left\Vert u-v\right\Vert _{\beta}. \label{f12}%
\end{align}
Here, in estimations (\ref{f1})-(\ref{f12}), we used the H\"{o}lder inequality%
\begin{align*}%
{\displaystyle\int\limits_{0}^{t}}
l_{f}\left(  s\right)  \left(  t-s\right)  ^{\alpha-m-1}ds  &  \leq\left(
{\displaystyle\int\limits_{0}^{t}}
\left(  l_{f}\left(  s\right)  \right)  ^{\frac{1}{\tau}}ds\right)  ^{\tau
}\left(
{\displaystyle\int\limits_{0}^{t}}
\left(  \left(  t-s\right)  ^{\alpha-m-1}\right)  ^{\frac{1}{1-\tau}%
}ds\right)  ^{1-\tau}\\
&  =\left\Vert l_{f}\right\Vert _{L^{1/\tau}}\left(  \frac{1-\tau}%
{\alpha-m-\tau}\right)  ^{1-\tau}t^{\alpha-m-\tau},\ \text{\ \ if }%
0<\gamma<\alpha-m.
\end{align*}
From (\ref{f1})-(\ref{f12}), it follows that%
\begin{align*}
&  \left\Vert \left(  \mathfrak{F}u\right)  -\left(  \mathfrak{F}v\right)
\right\Vert _{\beta}\leq\\
&  \leq\left[  \left(  \Delta_{0}+\Delta_{1}+\Delta_{2}\right)  \left\Vert
l_{f}\right\Vert _{1/\tau}+%
{\displaystyle\sum\limits_{i=0}^{2}}
\rho_{i}\left\vert \lambda_{i}\right\vert \left\Vert l_{g_{i}}\right\Vert
_{1}+%
{\displaystyle\sum\limits_{i=1}^{2}}
\widetilde{\rho}_{i}\left\vert \lambda_{i}\right\vert \left\Vert l_{g_{i}%
}\right\Vert _{1}+\widehat{\rho}_{2}\left\vert \lambda_{2}\right\vert
\left\Vert l_{g_{2}}\right\Vert _{1}\right]  \left\Vert u-v\right\Vert
_{\beta}.
\end{align*}
Consequently by (\ref{cc1}), $\mathfrak{F}$ is a contraction mapping. As a
consequence of the Banach fixed point theorem, we deduce that $\mathfrak{F}$
has a fixed point which is a solution of the problem (\ref{p1})-(\ref{p2}).
\end{proof}

\begin{remark}
In the assumptions (H$_{2}$) if $l_{f}$ is a constant then the condition
(\ref{cc1}) can be replaced by%
\begin{align*}
&  \dfrac{l_{f}T^{\alpha}}{\Gamma(\alpha+1)}+l_{f}%
{\displaystyle\sum\limits_{i=0}^{2}}
\rho_{i}\left(  \left\vert b_{i}\right\vert \dfrac{T^{\alpha-\beta_{i}}%
}{\Gamma(\alpha-\beta_{i}+1)}+\left\vert a_{i}\right\vert \dfrac{\eta
^{\alpha-\beta_{i}}}{\Gamma(\alpha-\beta_{i}+1)}\right) \\
&  +\dfrac{l_{f}T^{\alpha-\beta_{1}}}{\Gamma(\alpha-\beta_{1}+1)}+l_{f}%
{\displaystyle\sum\limits_{i=1}^{2}}
\widetilde{\rho}_{i}\left(  \left\vert b_{i}\right\vert \dfrac{T^{\alpha
-\beta_{i}}}{\Gamma(\alpha-\beta_{i}+1)}+\left\vert a_{i}\right\vert
\dfrac{\eta^{\alpha-\beta_{i}}}{\Gamma(\alpha-\beta_{i}+1)}\right) \\
&  +\dfrac{l_{f}T^{\alpha-\beta_{2}}}{\Gamma(\alpha-\beta_{2}+1)}%
+l_{f}\widehat{\rho}_{2}\left(  \left\vert b_{2}\right\vert \dfrac
{T^{\alpha-\beta_{2}}}{\Gamma(\alpha-\beta_{2}+1)}+\left\vert a_{2}\right\vert
\dfrac{\eta^{\alpha-\beta_{2}}}{\Gamma(\alpha-\beta_{2}+1)}\right) \\
&  +%
{\displaystyle\sum\limits_{i=0}^{2}}
\rho_{i}\left\vert \lambda_{i}\right\vert \left\Vert l_{g_{i}}\right\Vert
_{1}+%
{\displaystyle\sum\limits_{i=1}^{2}}
\widetilde{\rho}_{i}\left\vert \lambda_{i}\right\vert \left\Vert l_{g_{i}%
}\right\Vert _{1}+\widehat{\rho}_{2}\left\vert \lambda_{2}\right\vert
\left\Vert l_{g_{2}}\right\Vert _{1}<1.
\end{align*}

\end{remark}

\section{Existence results}

To prove the existence of solutions for BVP (\ref{p1})-(\ref{p2}), we recall
the following known nonlinear alternative.

\begin{theorem}
\label{Thm:na}(Nonlinear alternative) Let $X$ be a Banach space, let $B$ be a
closed, convex subset of $X$, let $W$ be an open subset of $B$ and $0\in W$.
Suppose that $F:\overline{W}\rightarrow B$ is a continuous and compact map.
Then either (a) $F$ has a fixed point in $\overline{W}$, or (b) there exist an
$x\in\partial W$ (the boundary of $W$) and $\lambda\in\left(  0,1\right)  $
with $x=\lambda F\left(  x\right)  .$
\end{theorem}

\begin{theorem}
\label{Thm:exis}Assume that

\begin{enumerate}
\item[(H$_{4}$)] there exist non-decreasing functions $\varphi:\left[
0,\infty\right)  \times\left[  0,\infty\right)  \times\left[  0,\infty\right)
\rightarrow\left[  0,\infty\right)  ,\ \psi_{i}:\left[  0,\infty\right)
\rightarrow\left[  0,\infty\right)  $ and functions $l_{f}\in L^{\frac{1}%
{\tau}}\left(  \left[  0,T\right]  ,\mathbb{R}^{+}\right)  $, $l_{g_{i}}\in
L^{1}\left(  \left[  0,T\right]  ,\mathbb{R}^{+}\right)  $ with $\tau
\in\left(  1,\min(\alpha-\beta_{2}\right)  )$ such that
\begin{align*}
\left\vert f\left(  t,u,v,w\right)  \right\vert  &  \leq l_{f}\left(
t\right)  \varphi\left(  \left\vert u\right\vert +\left\vert v\right\vert
+\left\vert w\right\vert \right)  ,\\
\left\vert g\left(  t,u\right)  \right\vert  &  \leq l_{g_{i}}\left(
t\right)  \psi_{i}\left(  \left\vert u\right\vert \right)
\end{align*}
$i=0,1,2$ for all $t\in\left[  0,T\right]  $ and $u,v,w\in\mathbb{R}$;

\item[(H$_{5}$)] there exists a constant $K>0$ such that%
\[
\frac{K}{\varphi\left(  K\right)  \left\Vert l_{f}\right\Vert _{1/\tau}\left(
\Delta_{0}+\Delta_{1}+\Delta_{2}\right)  +%
{\displaystyle\sum\limits_{i=0}^{2}}
\left(  \rho_{i}+\widetilde{\rho}_{i}+\widehat{\rho}_{i}\right)  \left\vert
\lambda_{i}\right\vert \psi_{i}\left(  K\right)  \left\Vert l_{g_{i}%
}\right\Vert _{1}}>1.
\]
Then the problem (\ref{p1})-(\ref{p2}) has at least one solution on $\left[
0,T\right]  .$
\end{enumerate}
\end{theorem}

\begin{proof}
Let $B_{r}:=\left\{  u\in C_{\beta}\left(  \left[  0,T\right]  ;\mathbb{R}%
\right)  :\left\Vert u\right\Vert _{\beta}\leq r\right\}  $.

Step 1: We show that the operator $\mathfrak{F:}C_{\beta}\left(  \left[
0,T\right]  ;\mathbb{R}\right)  \rightarrow C_{\beta}\left(  \left[
0,T\right]  ;\mathbb{R}\right)  $ defined by (\ref{ff1}) maps $B_{r}$ into
bounded set.

For each $u\in B_{r}$, we have%
\begin{align*}
\left\vert \left(  \mathfrak{F}u\right)  \left(  t\right)  \right\vert  &
\leq\dfrac{\varphi\left(  r\right)  }{\Gamma(\alpha)}%
{\displaystyle\int\limits_{0}^{t}}
(t-s)^{\alpha-1}\left\vert l_{f}(s)\right\vert ds\\
&  +\varphi\left(  r\right)
{\displaystyle\sum\limits_{i=0}^{2}}
\rho_{i}\left\vert b_{i}\right\vert \frac{1}{\Gamma(\alpha-\beta_{i})}%
{\displaystyle\int\limits_{0}^{T}}
(T-s)^{\alpha-\beta_{i}-1}\left\vert l_{f}(s)\right\vert ds\\
&  +\varphi\left(  r\right)
{\displaystyle\sum\limits_{i=1}^{2}}
\rho_{i}\left\vert a_{i}\right\vert \frac{1}{\Gamma(\alpha-\beta_{i})}%
{\displaystyle\int\limits_{0}^{\eta}}
(\eta-s)^{\alpha-\beta_{i}-1}\left\vert l_{f}(s)\right\vert ds\\
&  +%
{\displaystyle\sum\limits_{i=0}^{2}}
\rho_{i}\left\vert \lambda_{i}\right\vert \psi_{i}\left(  r\right)
{\displaystyle\int\limits_{0}^{T}}
\left\vert l_{g_{i}}(s)\right\vert ds.
\end{align*}
By the H\"{o}lder inequality, we have%
\begin{align*}
\left\vert \left(  \mathfrak{F}u\right)  \left(  t\right)  \right\vert  &
\leq\varphi\left(  r\right)  \left\Vert l_{f}\right\Vert _{1/\tau}\left(
\dfrac{T^{\alpha-\tau}}{\Gamma(\alpha)}\left(  \frac{1-\tau}{\alpha-\tau
}\right)  ^{1-\tau}+%
{\displaystyle\sum\limits_{i=0}^{2}}
\rho_{i}\left\vert b_{i}\right\vert \dfrac{T^{\alpha-\beta_{i}-\tau}}%
{\Gamma(\alpha-\beta_{i})}\left(  \frac{1-\tau}{\alpha-\beta_{i}-\tau}\right)
^{1-\tau}\right. \\
&  +\left.
{\displaystyle\sum\limits_{i=1}^{2}}
\rho_{i}\left\vert a_{i}\right\vert \dfrac{\eta^{\alpha-\beta_{i}-\tau}%
}{\Gamma(\alpha-\beta_{i})}\left(  \frac{1-\tau}{\alpha-\beta_{i}-\tau
}\right)  ^{1-\tau}\right)  +%
{\displaystyle\sum\limits_{i=0}^{2}}
\rho_{i}\left\vert \lambda_{i}\right\vert \psi_{i}\left(  r\right)  \left\Vert
l_{g_{i}}\right\Vert _{1}\\
&  =\varphi\left(  r\right)  \left\Vert l_{f}\right\Vert _{1/\tau}\Delta_{0}+%
{\displaystyle\sum\limits_{i=0}^{2}}
\rho_{i}\left\vert \lambda_{i}\right\vert \psi_{i}\left(  r\right)  \left\Vert
l_{g_{i}}\right\Vert _{1}.
\end{align*}
In a similar manner,%
\begin{align*}
\left\vert \mathfrak{D}_{0+}^{\beta_{1}}\left(  \mathfrak{F}u\right)  \left(
t\right)  \right\vert  &  \leq\varphi\left(  r\right)  \left\Vert
l_{f}\right\Vert _{1/\tau}\left(  \dfrac{T^{\alpha-\beta_{1}-\tau}}%
{\Gamma(\alpha-\beta_{1})}\left(  \frac{1-\tau}{\alpha-\beta_{1}-\tau}\right)
^{1-\tau}+%
{\displaystyle\sum\limits_{i=1}^{2}}
\widetilde{\rho}_{i}\left\vert b_{i}\right\vert \dfrac{T^{\alpha-\beta
_{i}-\tau}}{\Gamma(\alpha-\beta_{i})}\left(  \frac{1-\tau}{\alpha-\beta
_{i}-\tau}\right)  ^{1-\tau}\right. \\
&  +\left.
{\displaystyle\sum\limits_{i=1}^{2}}
\widetilde{\rho}_{i}\left\vert a_{i}\right\vert \dfrac{\eta^{\alpha-\beta
_{i}-\tau}}{\Gamma(\alpha-\beta_{i})}\left(  \frac{1-\tau}{\alpha-\beta
_{i}-\tau}\right)  ^{1-\tau}\right)  +%
{\displaystyle\sum\limits_{i=1}^{2}}
\widetilde{\rho}_{i}\left\vert \lambda_{i}\right\vert \psi_{i}\left(
r\right)  \left\Vert l_{g_{i}}\right\Vert _{1}\\
&  =\varphi\left(  r\right)  \left\Vert l_{f}\right\Vert _{1/\tau}\Delta_{1}+%
{\displaystyle\sum\limits_{i=1}^{2}}
\widetilde{\rho}_{i}\left\vert \lambda_{i}\right\vert \psi_{i}\left(
r\right)  \left\Vert l_{g_{i}}\right\Vert _{1},
\end{align*}
and%
\begin{align*}
\left\vert \mathfrak{D}_{0+}^{\beta_{2}}\left(  \mathfrak{F}u\right)  \left(
t\right)  \right\vert  &  \leq\varphi\left(  r\right)  \left\Vert
l_{f}\right\Vert _{1/\tau}\left(  \dfrac{T^{\alpha-\beta_{2}-\tau}}%
{\Gamma(\alpha-\beta_{2})}\left(  \frac{1-\tau}{\alpha-\beta_{2}-\tau}\right)
^{1-\tau}+\widehat{\rho}_{2}\frac{T^{\alpha-\beta_{2}-\tau}\left\vert
b_{2}\right\vert }{\Gamma(\alpha-\beta_{2})}\left(  \frac{1-\tau}{\alpha
-\beta_{2}-\tau}\right)  ^{1-\tau}\right. \\
&  +\left.  \widehat{\rho}_{2}\frac{\eta^{\alpha-\beta_{2}-\tau}\left\vert
b_{2}\right\vert }{\Gamma(\alpha-\beta_{2})}\left(  \frac{1-\tau}{\alpha
-\beta_{2}-\tau}\right)  ^{1-\tau}\right)  +\widehat{\rho}_{2}\left\vert
\lambda_{2}\right\vert \psi_{2}\left(  r\right)  \left\Vert l_{g_{2}%
}\right\Vert _{1}\\
&  =\varphi\left(  r\right)  \left\Vert l_{f}\right\Vert _{1/\tau}\Delta
_{2}+\widehat{\rho}_{2}\left\vert \lambda_{2}\right\vert \psi_{2}\left(
r\right)  \left\Vert l_{g_{2}}\right\Vert _{1}.
\end{align*}
Thus%
\[
\left\Vert \left(  \mathfrak{F}u\right)  \right\Vert _{\beta}\leq
\varphi\left(  r\right)  \left\Vert l_{f}\right\Vert _{1/\tau}\left(
\Delta_{0}+\Delta_{1}+\Delta_{2}\right)  +%
{\displaystyle\sum\limits_{i=0}^{2}}
\left(  \rho_{i}+\widetilde{\rho}_{i}+\widehat{\rho}_{i}\right)  \left\vert
\lambda_{i}\right\vert \psi_{i}\left(  r\right)  \left\Vert l_{g_{i}%
}\right\Vert _{1}.
\]

Step 2: The families $\left\{  \mathfrak{F}u:u\in B_{r}\right\}  $, $\left\{
\mathfrak{D}_{0+}^{\beta_{1}}\left(  \mathfrak{F}u\right)  :u\in
B_{r}\right\}  $, $\left\{  \mathfrak{D}_{0+}^{\beta_{2}}\left(
\mathfrak{F}u\right)  :u\in B_{r}\right\}  $ are equicontinuous.

Because of continuity of $\omega_{i}\left(  t\right)  $ and assumption
(H$_{4}$) we have
\begin{align*}
\left\vert \left(  \mathfrak{F}u\right)  \left(  t_{2}\right)  -\left(
\mathfrak{F}u\right)  \left(  t_{1}\right)  \right\vert  &  \leq\dfrac
{1}{\Gamma(\alpha)}\varphi\left(  r\right)  \int_{t_{1}}^{t_{2}}%
(t_{2}-s)^{\alpha-1}l_{f}\left(  s\right)  ds\\
&  +\dfrac{1}{\Gamma(\alpha)}\varphi\left(  r\right)  \int_{0}^{t_{1}}\left(
(t_{2}-s)^{\alpha-1}-(t_{1}-s)^{\alpha-1}\right)  l_{f}\left(  s\right)  ds\\
&  +\varphi\left(  r\right)
{\displaystyle\sum\limits_{i=0}^{2}}
\left\vert \omega_{i}\left(  t_{2}\right)  -\omega_{i}\left(  t_{1}\right)
\right\vert \left\vert b_{i}\right\vert
{\displaystyle\int\limits_{0}^{T}}
\frac{(T-s)^{\alpha-\beta_{i}-1}}{\Gamma(\alpha-\beta_{i})}l_{f}\left(
s\right)  ds\\
&  +\varphi\left(  r\right)
{\displaystyle\sum\limits_{i=1}^{2}}
\left\vert \omega_{i}\left(  t_{2}\right)  -\omega_{i}\left(  t_{1}\right)
\right\vert \left\vert a_{i}\right\vert
{\displaystyle\int\limits_{0}^{\eta}}
\frac{(\eta-s)^{\alpha-\beta_{i}-1}}{\Gamma(\alpha-\beta_{i})}l_{f}\left(
s\right)  ds\\
&  +%
{\displaystyle\sum\limits_{i=0}^{2}}
\left\vert \omega_{i}\left(  t_{2}\right)  -\omega_{i}\left(  t_{1}\right)
\right\vert \left\vert \lambda_{i}\right\vert \psi_{i}\left(  r\right)
\left\Vert l_{g_{i}}\right\Vert _{1}\\
&  \rightarrow0\ \ \ \ \ \text{as\ \ \ }t_{2}\rightarrow t_{1}.
\end{align*}
Therefore, $\left\{  \mathfrak{F}u:u\in B_{r}\right\}  $ is equicontinuous.
Similarly, we may prove that $\left\{  \mathfrak{D}_{0+}^{\beta_{1}}\left(
\mathfrak{F}u\right)  :u\in B_{r}\right\}  $ and $\left\{  \mathfrak{D}%
_{0+}^{\beta_{2}}\left(  \mathfrak{F}u\right)  :u\in B_{r}\right\}  $ are equicontinuous.

Hence, by the Arzela--Ascoli theorem, the sets $\left\{  \mathfrak{F}u:u\in
B_{r}\right\}  $, $\left\{  \mathfrak{D}_{0+}^{\beta_{1}}\left(
\mathfrak{F}u\right)  :u\in B_{r}\right\}  ,$ $\left\{  \mathfrak{D}%
_{0+}^{\beta_{2}}\left(  \mathfrak{F}u\right)  :u\in B_{r}\right\}  $ are
relatively compact in $C\left(  \left[  0,T\right]  ;\mathbb{R}\right)  .$
Therefore, $\mathfrak{F}\left(  B_{r}\right)  $ is a relatively compact subset
of $C_{\beta}\left(  \left[  0,T\right]  ;\mathbb{R}\right)  .$ Consequently,
the operator $\mathfrak{F}$ is compact.

Step 3: $\mathfrak{F}$ has a fixed in $\overline{W}=\left\{  u\in C_{\beta
}\left(  \left[  0,T\right]  ;\mathbb{R}\right)  :\left\Vert u\right\Vert
_{\beta}<K\right\}  .$

We let $u=\lambda\left(  \mathfrak{F}u\right)  $ for $0<\lambda<1$. Then for
each $t\in\left[  0,T\right]  $,
\begin{align*}
\left\Vert u\right\Vert _{\beta}  &  =\left\Vert \lambda\left(  \mathfrak{F}%
u\right)  \right\Vert _{\beta}\leq\varphi\left(  \left\Vert u\right\Vert
_{\beta}\right)  \left\Vert l_{f}\right\Vert _{1/\tau}\left(  \Delta
_{0}+\Delta_{1}+\Delta_{2}\right) \\
&  +%
{\displaystyle\sum\limits_{i=0}^{2}}
\left(  \rho_{i}+\widetilde{\rho}_{i}+\widehat{\rho}_{i}\right)  \left\vert
\lambda_{i}\right\vert \psi_{i}\left(  \left\Vert u\right\Vert _{\beta
}\right)  \left\Vert l_{g_{i}}\right\Vert _{1}.
\end{align*}
In other words,%
\[
\frac{\left\Vert u\right\Vert _{\beta}}{\varphi\left(  \left\Vert u\right\Vert
_{\beta}\right)  \left\Vert l_{f}\right\Vert _{1/\tau}\left(  \Delta
_{0}+\Delta_{1}+\Delta_{2}\right)  +%
{\displaystyle\sum\limits_{i=0}^{2}}
\left(  \rho_{i}+\widetilde{\rho}_{i}+\widehat{\rho}_{i}\right)  \left\vert
\lambda_{i}\right\vert \psi_{i}\left(  \left\Vert u\right\Vert _{\beta
}\right)  \left\Vert l_{g_{i}}\right\Vert _{1}}\leq1.
\]
According to the assumptions, we know that there exists $K>0$ such that
$K\neq\left\Vert u\right\Vert _{\beta}$. The operator $\mathfrak{F:}%
\overline{W}$ $\rightarrow C_{\beta}\left(  \left[  0,T\right]  ;\mathbb{R}%
\right)  $ is continuous and compact. From Theorem \ref{Thm:na}, we can deduce
that $\mathfrak{F}$ has a fixed point in $\overline{W}$.
\end{proof}

\begin{remark}
Notice that analogues of Theorem \ref{Thm:uniq} and \ref{Thm:exis} for the
case $f(t,u,v,w)=f(t,u)$ were considered in \cite{ahmadntoy2}. Thus our
results are generalization of \cite{ahmadntoy2}.
\end{remark}

\begin{remark}
Since the number $\left(  \alpha-\beta_{2}-1\right)  $ can be negative, the
function $(T-s)^{\alpha-\beta_{2}-1}\notin L^{\infty}\left(  \left[
0,T\right]  ,\mathbb{R}\right)  $. That is why in Theorem \ref{Thm:uniq} and
\ref{Thm:exis} it is assumed that $l_{f}\in L^{\frac{1}{\tau}}$, $\tau
\in(0,\min(1,\alpha-\beta_{2})).$
\end{remark}

\section{Examples}

\textbf{Example 1.} Consider the following boundary value problem of
fractional differential equation:%
\begin{equation}
\left\{
\begin{array}
[c]{c}%
\mathfrak{D}_{0^{+}}^{5/2}u\left(  t\right)  =l_{f}\left(  \dfrac{\left\vert
u\left(  t\right)  \right\vert }{1+\left\vert u\left(  t\right)  \right\vert
}+\dfrac{\left\vert \mathfrak{D}_{0^{+}}^{1/2}u\left(  t\right)  \right\vert
}{1+\left\vert \mathfrak{D}_{0^{+}}^{1/2}u\left(  t\right)  \right\vert }%
+\tan^{-1}\left(  \mathfrak{D}_{0^{+}}^{3/2}u\left(  t\right)  \right)
\right)  ,\ \ \ 0\leq t\leq1,\\
u\left(  0\right)  +u\left(  1\right)  =%
{\displaystyle\int_{0}^{1}}
\dfrac{u\left(  s\right)  }{\left(  1+s\right)  ^{2}}ds,\\
\mathfrak{D}_{0^{+}}^{1/2}u\left(  \frac{1}{10}\right)  +\mathfrak{D}_{0^{+}%
}^{1/2}u\left(  1\right)  =\dfrac{1}{2}%
{\displaystyle\int_{0}^{1}}
\left(  \dfrac{e^{s}u\left(  s\right)  }{1+2e^{s}}+\dfrac{1}{2}\right)  ds,\\
\mathfrak{D}_{0^{+}}^{3/2}u\left(  \frac{1}{10}\right)  +\mathfrak{D}_{0^{+}%
}^{3/2}u\left(  1\right)  =\dfrac{1}{3}%
{\displaystyle\int_{0}^{1}}
\left(  \dfrac{u\left(  s\right)  }{1+e^{s}}+\dfrac{3}{4}\right)  ds.
\end{array}
\right.  \label{ex1}%
\end{equation}
Here
\begin{align*}
\alpha &  =5/2,\beta_{1}=1/2,\beta_{2}=3/2,T=1,a_{0}=b_{0}=a_{1}=b_{1}%
=a_{2}=b_{2}=1,\ \\
\eta &  =\frac{1}{10}\ \ \lambda_{0}=1,\ \ \lambda_{1}=\frac{1}{2}%
,\ \ \lambda_{2}=\frac{1}{3},\ l_{g_{0}}=l_{g_{1}}=l_{g_{2}}=1,
\end{align*}
and
\begin{align*}
f\left(  t,u,v,w\right)   &  :=\dfrac{u}{1+u}+\dfrac{v}{1+v}+\tan^{-1}\left(
w\right)  ,\\
g_{0}\left(  t,u\right)   &  :=\dfrac{u}{\left(  1+t\right)  ^{2}}%
,\ \ \ g_{1}\left(  t,u\right)  :=\dfrac{e^{t}u}{1+2e^{t}}+\frac{1}{2},\\
g_{2}\left(  t,u\right)   &  :=\dfrac{u}{1+e^{t}}+\frac{3}{4}.
\end{align*}
Since $1.77<\Gamma(\frac{1}{2})<1.78;0.88<\Gamma(\frac{3}{2})<0.89;1.32<\Gamma
(\frac{5}{2})<1.33$ and $3.32<\Gamma(\frac{7}{2})<3.33$ with simple
calculations we show that%
\begin{align*}
\Delta_{0}  &  =2.34,\ \ \Delta_{1}=0.19,\ \ \Delta_{2}=0.15,\\
\rho_{0}  &  =0.5,\ \ \ \rho_{1}=1.01,\ \ \ \rho_{2}=1.2,\ \ \\
\tilde{\rho}_{0}  &  =0,\ \ \tilde{\rho}_{1}=0.76,\ \ \ \tilde{\rho}%
_{2}=0.9,\ \ \\
\hat{\rho}_{0}  &  =\hat{\rho}_{1}=0,\text{ \ }\hat{\rho}_{2}=0.51
\end{align*}
Furthermore,
\begin{align*}
&  \left(  \Delta_{0}+\Delta_{1}+\Delta_{2}\right)  \left\Vert l_{f}%
\right\Vert _{1/\tau}+%
{\displaystyle\sum\limits_{i=0}^{2}}
\rho_{i}\left\vert \lambda_{i}\right\vert \left\Vert l_{g_{i}}\right\Vert
_{1}+%
{\displaystyle\sum\limits_{i=1}^{2}}
\widetilde{\rho}_{i}\left\vert \lambda_{i}\right\vert \left\Vert l_{g_{i}%
}\right\Vert _{1}+\widehat{\rho}_{2}\left\vert \lambda_{2}\right\vert
\left\Vert l_{g_{2}}\right\Vert _{1}\\
&  <2.7l_{f}+0.75<1.
\end{align*}
Therefore, we can choose
\[
l_{f}<\frac{0.25}{2.7}.
\]
Thus, all the assumptions of Theorem \ref{Thm:uniq} are satisfied. Hence, the
problem (\ref{ex1}) has a unique solution on $[0,1]$.

\textbf{Example 2. \ }Consider the following boundary value problem of
fractional differential equation:%
\begin{equation}
\left\{
\begin{array}
[c]{c}%
\mathfrak{D}_{0^{+}}^{5/2}u(t)=\dfrac{\left\vert u\left(  t\right)
\right\vert ^{3}}{9(\left\vert u\left(  t\right)  \right\vert ^{3}+3)}%
+\dfrac{\left\vert \sin\mathfrak{D}_{0^{+}}^{1/2}u\left(  t\right)
\right\vert }{9(\left\vert \sin\mathfrak{D}_{0^{+}}^{1/2}u\left(  t\right)
\right\vert +1)}+\dfrac{1}{12},\ \ t\in\left[  0,1\right]  ,\\
u(0)+u(1)=%
{\displaystyle\int\limits_{0}^{1}}
\dfrac{u(s)}{3(1+s)^{2}}ds,\\
\mathfrak{D}_{0^{+}}^{1/2}u\left(  \frac{1}{10}\right)  +\mathfrak{D}_{0^{+}%
}^{1/2}u\left(  1\right)  =\dfrac{1}{2}%
{\displaystyle\int\limits_{0}^{1}}
\dfrac{e^{s}u(s)}{3(1+e^{s})^{2}}ds,\\
\mathfrak{D}_{0^{+}}^{3/2}u\left(  \frac{1}{10}\right)  +\mathfrak{D}_{0^{+}%
}^{3/2}u\left(  1\right)  =\dfrac{1}{3}%
{\displaystyle\int\limits_{0}^{1}}
\dfrac{u(s)}{3(1+e^{s})^{2}}ds,
\end{array}
\right.  \label{ex2}%
\end{equation}
where $f$ is given by%

\[
f(t,u,v,w)=\frac{\left\vert u\right\vert ^{3}}{10(\left\vert u\right\vert
^{3}+3)}+\frac{\left\vert \sin v\right\vert }{9(\left\vert \sin v\right\vert
+1)}+\frac{1}{12}.
\]
We have%
\[
\left\vert f(t,u,v,w)\right\vert \leq\frac{\left\vert u\right\vert ^{3}%
}{9(\left\vert u\right\vert ^{3}+3)}+\frac{\left\vert \sin v\right\vert
}{9(\left\vert \sin v\right\vert +1)}+\frac{1}{12}\leq\frac{11}{36},\text{
\ \ \ \ }u\in\mathbb{R}.
\]
Thus%
\[
\left\Vert f\right\Vert \leq\frac{11}{36}=l_{f}(t)\varphi\left(  K\right)
,\ \ \ \ \text{with \ \ }l_{f}(t)=\dfrac{1}{3},\varphi\left(  K\right)
=\dfrac{11}{12}.
\]
Moreover
\begin{align*}
\alpha &  =5/2,\beta_{1}=1/2,\beta_{2}=3/2,T=1,a_{0}=b_{0}=a_{1}=b_{1}%
=a_{2}=b_{2}=1,\ \\
\eta &  =\frac{1}{10}\ \ \lambda_{0}=1,\ \ \lambda_{1}=\frac{1}{2}%
,\ \ \lambda_{2}=\frac{1}{3},\ l_{g_{0}}=l_{g_{1}}=l_{g_{2}}=\frac{1}{3},
\end{align*}%
\begin{align*}
\Delta_{0}  &  =2.34,\ \ \Delta_{1}=0.19,\ \ \Delta_{2}=0.15,\\
\rho_{0}  &  =0.5,\ \ \ \rho_{1}=1.01,\ \ \ \rho_{2}=1.2,\ \ \\
\tilde{\rho}_{0}  &  =0,\ \ \tilde{\rho}_{1}=0.76,\ \ \ \tilde{\rho}%
_{2}=0.9,\ \ \\
\hat{\rho}_{0}  &  =\hat{\rho}_{1}=0,\text{ \ }\hat{\rho}_{2}=0.51,
\end{align*}
and
\[
g_{0}\left(  t,u\right)  :=\dfrac{u}{3(1+t)^{2}},\ \ \ g_{1}\left(
t,u\right)  :=\dfrac{e^{t}u}{3(1+e^{t})^{2}},\ \ g_{2}\left(  t,u\right)
:=\dfrac{u}{3(1+e^{t})^{2}},\ \ \psi_{i}\left(  K\right)  =K.
\]
From the condition
\[
\frac{K}{\varphi\left(  K\right)  \left\Vert l_{f}\right\Vert _{1/\tau}\left(
\Delta_{0}+\Delta_{1}+\Delta_{2}\right)  +%
{\displaystyle\sum\limits_{i=0}^{2}}
\left(  \rho_{i}+\widetilde{\rho}_{i}+\widehat{\rho}_{i}\right)  \left\vert
\lambda_{i}\right\vert \psi_{i}\left(  K\right)  \left\Vert l_{g_{i}%
}\right\Vert _{1}}>1
\]
we find that%
\[
K>9.8.
\]
Thus, all the conditions of Theorem \ref{Thm:exis} are satisfied. So, there
exists at least one solution of problem (\ref{ex2}) on $\left[  0,1\right]  $.

\end{document}